\documentclass[reqno]{amsart}
\usepackage{amssymb,mathtools,lineno,mathrsfs}

\usepackage{ifpdf}
\ifpdf
 \usepackage[hyperindex]{hyperref}
\else
 \expandafter\ifx\csname dvipdfm\endcsname\relax
 \usepackage[hypertex,hyperindex]{hyperref}%
 \else
 \usepackage[dvipdfm,hyperindex]{hyperref}%
 \fi
\fi

\numberwithin{equation}{section}
\allowdisplaybreaks[4]

\theoremstyle{plain}
\newtheorem{thm}{Theorem}[section]
\newtheorem*{thm1st}{Gauss' first summation theorem}
\newtheorem*{thm2nd}{Gauss' second summation theorem}

\theoremstyle{remark}
\newtheorem{rem}{Remark}[section]

\DeclareMathOperator{\te}{e}

\begin{document}

\title[Extensions of four famous combinatorial identities]
{Extensions of several famous combinatorial identities via hypergeometric functions}

\author[A. K. Rathie]{Arjun Kumar Rathie}
\address{Department of Mathematics, Vedant College of Engineering and Technololy, (Rajasthan Technical University), Tulsi, Jakhamund, Bundi Rajasthan State, India}
\email{arjunkumarrathie@gmail.com}
\urladdr{\url{https://orcid.org/0000-0003-3902-3050}}

\author[F. Qi]{Feng Qi}
\address{School of Mathematics and Physics, Hulunbuir University, Hulunbuir 021008, Inner Mongolia, China;
17709 Sabal Court, University Village, Dallas, TX 75252-8024, USA}
\email{\href{mailto: F. Qi<qifeng618@gmail.com>}{qifeng618@gmail.com}}
\urladdr{\url{https://orcid.org/0000-0001-6239-2968}}

\author[D. Lim]{Dongkyu Lim*}
\address{Department of Mathematics Education, Gyeongkuk National University, Andong 36729, Republic of Korea}
\email{dklim@gknu.ac.kr}
\urladdr{\url{https://orcid.org/0000-0002-0928-8480}} 

\begin{abstract}
The objective of this paper is to develop an extension of Kummer's second theorem and to establish generalized forms of four classical combinatorial identities---Knuth's old sum (also known as Reed--Dawson's combinatorial identity), Riordan's combinatorial identity, Gould's combinatorial identity, and Touchard's combinatorial identity---using a hypergeometric-series approach. Several new identities also arise as special cases of our main results.
\end{abstract}



\thanks{*Corresponding author}

\thanks{This paper was typeset using \AmS-\LaTeX}

\maketitle

\section{Introduction}
In this section, we review the required foundational concepts and outline several pertinent theorems and combinatorial identities.

\subsection{Generalized hypergeometric functions}
As usual, let $\mathbb{C}$ denote the set of complex numbers and $\mathbb{Z}_0^-=\{0, -1, -2, \dotsc\}$. For $\alpha_j \in \mathbb{C}$ and $\beta_j \in \mathbb{C} \setminus \mathbb{Z}_0^-$, the generalized hypergeometric function ${}_pF_q$ with $p$ numerator parameters $\alpha_1,\alpha_2, \dotsc, \alpha_p$ and $q$ denominator parameters $\beta_1, \beta_2,\dotsc, \beta_q$ for $p, q \in \mathbb{N}_0= \{0,1,2,\dotsc\}$ is defined~\cite{23} by
\begin{equation*}
{\,}_pF_q\begin{bmatrix}\begin{matrix}
\alpha_1, \alpha_2,\dotsc, \alpha_p\\
\beta_1, \beta_2,\dotsc, \beta_q
\end{matrix};z\end{bmatrix}
= \sum_{n=0}^{\infty}
\frac{(\alpha_1)_n (\alpha_2)_n \cdots (\alpha_p)_n}
{(\beta_1)_n (\beta_2)_n \cdots (\beta_q)_n}
\frac{z^n}{n!},\quad |z|<1,
\end{equation*}
where $(\alpha)_n$ denotes the Pochhammer symbol (or shifted factorial) which is defined~\cite[p.~72]{Temme-96-book} for $\alpha \in \mathbb{C}$ by
\[
(\alpha)_n=\frac{\Gamma(\alpha+n)}{\Gamma(\alpha)} =
\begin{cases}
1, & n=0;\\
\alpha(\alpha+1)\cdots(\alpha+n - 1), & n \in \mathbb{N},
\end{cases}
\]
and the classical Euler gamma function $\Gamma(z)$ can be defined~\cite[Chapter~3]{Temme-96-book} by
\begin{equation*}
\Gamma(z)=\lim_{n\to\infty}\frac{n!n^z}{\prod_{k=0}^n(z+k)}, \quad z\in\mathbb{C}\setminus\mathbb{Z}_0^-.
\end{equation*}
In particular, ${\,}_2F_1$ is the classical Gauss hypergeometric function~\cite[Chapter~5]{Temme-96-book}.
\par
The theory of the Gauss and generalized hypergeometric functions ${}_pF_q$ plays a fundamental role in both mathematics and mathematical physics. Many of the elementary functions that arise in analysis can in fact be viewed as limiting cases or special instances of hypergeometric functions ${}_pF_q$.
\par
It is worth noting that whenever Gauss' hypergeometric function ${}_2F_1$ or, more generally, the generalized hypergeometric function ${}_pF_q$ reduces to expressions involving the gamma functions $\Gamma(z)$, the resulting formulas are of considerable importance from the perspective of applied considerations. Accordingly, the classical summation theorems for ${}_2F_1$, ${}_3F_2$, ${}_4F_3$, and related functions, as well as the associated transformation formulas, play a central role. 

\subsection{Gauss' summation theorems}
In the present investigation, we concentrate on the following two summation formulas for the function ${}_2F_1$, originally due to Gauss~\cite{7,23}.

\begin{thm1st}
If $\Re{(c-a-b)}>0$, then
\begin{equation}\label{2}
{\,}_2F_1\begin{bmatrix}
\begin{gathered}
a,b\\ c
\end{gathered};1\end{bmatrix}
=\frac{\Gamma(c)\Gamma(c-a-b)}{\Gamma(c-a)\Gamma(c-b)}.
\end{equation}
\end{thm1st}

\begin{thm2nd}
If $\frac{a+b+1}{2}\not\in\mathbb{Z}_0^-$, then
\begin{equation}\label{3}
{\,}_2F_1\begin{bmatrix}
\begin{gathered}
a,b\\ \tfrac{a+b+1}{2}\end{gathered};\frac{1}{2}\end{bmatrix}
=\frac{\Gamma\bigl(\frac{1}{2}\bigr) \Gamma\bigl(\frac{a+b+1}{2}\bigr)} {\Gamma\bigl(\frac{a+1}{2}\bigr) \Gamma\bigl(\frac{b+1}{2}\bigr)}.
\end{equation}
\end{thm2nd}

Gauss' summation theorem~\eqref{2} has been extended~\cite{22} as
\begin{equation}\label{4}
{\,}_3F_2\begin{bmatrix}
\begin{gathered}
a,b,d+1\\ c+1,d
\end{gathered};1
\end{bmatrix}
=\biggl(c-a-b+\frac{ab}{d}\biggr) \frac{\Gamma\bigl(c+1\bigr)\Gamma(c-a-b)}{\Gamma\bigl(c-a+1\bigr)\Gamma\bigl(c-b+1\bigr)}
\end{equation}
provided $\Re{(c-a-b)}>0$ and $d\neq 0,-1,-2,\dotsc$.
When $d=c$, the extension~\eqref{4} reduces to~\eqref{2}.

\subsection{Kummer's second theorem}
Kummer's second theorem~\cite[p.~126]{23} reads that
\begin{equation}\label{5}
\te^{-x/2}{\,}_1F_1\begin{bmatrix}
\begin{gathered}
a\\ 2a
\end{gathered};x
\end{bmatrix}
={\,}_0F_1\begin{bmatrix}
\begin{gathered}\\ a+\tfrac{1}{2}
\end{gathered};
\dfrac{x^2}{16}
\end{bmatrix}.
\end{equation}
From~\eqref{5}, one can deduce
\begin{align}\label{6}
{\,}_2F_1\begin{bmatrix}
\begin{gathered}
-2n, a\\ 2a
\end{gathered}
;2\end{bmatrix}
&=\frac{\bigl(\frac{1}{2}\bigr)_n}{\bigl(a+\frac{1}{2}\bigr)_n}
\intertext{and}
\label{7}
{\,}_2F_1\begin{bmatrix}
\begin{gathered}
-2n-1, a\\2a
\end{gathered};2\end{bmatrix}&=0
\end{align}
for $n\in\mathbb{N}_0$; see~\cite[pp.~126--127, Eqs.~(2) and~(10)]{23}.
\par
In the year 1995, Rathie and Nagar~\cite{25} established the following two results closely-related to~\eqref{5}:
\begin{align}\label{8}
\te^{-x/2}{\,}_1F_1\begin{bmatrix}
\begin{gathered}
a\\ 2a+1
\end{gathered};x
\end{bmatrix}
&={\,}_0F_1\begin{bmatrix}
\begin{gathered}
\\ a+\tfrac{1}{2}
\end{gathered};\dfrac{x^2}{16}\end{bmatrix}
-\frac{x}{2(2a+1)}{\,}_0F_1\begin{bmatrix}
\begin{gathered}\\ a+\tfrac{3}{2}\end{gathered};\dfrac{x^2}{16}\end{bmatrix}
\intertext{and}
\label{1.7a}
\te^{-x/2}{\,}_1F_1\begin{bmatrix}
\begin{gathered}a\\ 2a-1\end{gathered};x\end{bmatrix}
&={\,}_0F_1\begin{bmatrix}\begin{gathered}\\a-\tfrac12\end{gathered};\dfrac{x^2}{16}\end{bmatrix}
+\frac{x}{2(2a-1){\,}_0F_1}\begin{bmatrix}\begin{gathered}\\a+\tfrac12\end{gathered};\dfrac{x^2}{16}\end{bmatrix}.
\end{align}
From~\eqref{8} and~\eqref{1.7a}, one can deduce
\begin{align}\label{9}
{\,}_2F_1\begin{bmatrix}
\begin{gathered}
-2n, a\\ 2a+1
\end{gathered};2\end{bmatrix}
&=\frac{\bigl(\frac{1}{2}\bigr)_n}{\bigl(a+\frac{1}{2}\bigr)_n},\\
\label{10}
{\,}_2F_1\begin{bmatrix}\begin{gathered}-2n-1, a\\ 2a+1\end{gathered};2\end{bmatrix}
&=\frac{\bigl(\frac{3}{2}\bigr)_n}{(2a+1)\bigl(a+\frac{3}{2}\bigr)_n},\\
{\,}_2F_1\begin{bmatrix}\begin{gathered}-2n, a\\ 2a-1\end{gathered};2\end{bmatrix} &=\frac{\bigl(\frac{1}{2}\bigr)_n}{\bigl(a-\frac{1}{2}\bigr)_n},\notag
\intertext{and}
{\,}_2F_1\begin{bmatrix}\begin{gathered}-2n-1, a\\ 2a-1\end{gathered};2\end{bmatrix} &=-\frac{1}{2a-1}\frac{\bigl(\frac{3}{2}\bigr)_n}{\bigl(a+\frac{1}{2}\bigr)_n}.\notag
\end{align}
for $n\in\mathbb{N}_0$; see also~\cite{16,17}.
\par
In 2008, Rathie and Pog\'any~\cite{26} established a natural extension of Kummer's second theorem~\eqref{5} as
\begin{equation}\label{14}
\te^{-x/2} {\,}_2F_2\begin{bmatrix}\begin{gathered}a,d+1\\ 2a+1,d\end{gathered};x\end{bmatrix}
={\,}_0F_1\begin{bmatrix}\begin{gathered}\\ a+\tfrac{1}{2}\end{gathered};\dfrac{x^2}{16}\end{bmatrix}
+\frac{x(2a-d)}{2d(2a+1)}{\,}_0F_1\begin{bmatrix}\begin{gathered}\\ a+\tfrac{3}{2}\end{gathered};\dfrac{x^2}{16}\end{bmatrix}.
\end{equation}
When $d=2a$, the extension~\eqref{14} reduces to~\eqref{5}.
\par
In 2010, Kim et al.~\cite{15} generalized Kummer's second theorem~\eqref{5}, obtained explicit expressions of
\begin{equation*}
\te^{-x/2}{\,}_1F_1\begin{bmatrix}
\begin{gathered}
a\\ 2a+k\end{gathered};x\end{bmatrix}, \quad k=0,\pm1, \pm2,\dotsc, \pm5,
\end{equation*}
and discussed some applications. One of those results in~\cite{15} is
\begin{multline}\label{11}
\te^{-x/2}{\,}_1F_1\begin{bmatrix}\begin{gathered}
a\\ 2a+2\end{gathered};x\end{bmatrix}
={\,}_0F_1\begin{bmatrix}\begin{gathered}\\ a+\tfrac{3}{2}\end{gathered};\dfrac{x^2}{16}\end{bmatrix}\\
-\frac{x}{2(a+1)}{\,}_0F_1\begin{bmatrix}\begin{gathered}\\ a+\tfrac{3}{2}\end{gathered};\dfrac{x^2}{16}\end{bmatrix}
+\frac{x^2}{4(a+1)(2a+3)}{\,}_0F_1\begin{bmatrix}\begin{gathered}\\ a+\tfrac{5}{2}\end{gathered};\dfrac{x^2}{16}\end{bmatrix}.
\end{multline}
In the same paper~\cite{15}, from~\eqref{11}, they obtained
\begin{align*}%
{\,}_2F_1\begin{bmatrix}\begin{gathered}-2n, a\\ 2a+2\end{gathered};2\end{bmatrix}
&=\frac{\bigl(\frac{1}{2}\bigr)_n\bigl(\frac{a+3}{2}\bigr)_n}{\bigl(a+\frac{3}{2}\bigr)_n\bigl(\frac{a+1}{2}\bigr)_n}
\intertext{and}
{\,}_2F_1\begin{bmatrix}\begin{gathered}-2n-1, a\\ 2a+2\end{gathered};2\end{bmatrix}
&=\frac{\bigl(\frac{3}{2}\bigr)_n}{\bigl(a+1\bigr)\bigl(a+\frac{3}{2}\bigr)_n}
\end{align*}
for $n\in\mathbb{N}_0$.
\par
In 2012, from the extension~\eqref{14}, Kim et al.~\cite{14} derived
\begin{align}\label{15}
{\,}_3F_2\begin{bmatrix}\begin{gathered}-2n, a, d+1\\ 2a+1, d \end{gathered};2\end{bmatrix}
&=\frac{\bigl(\frac{1}{2}\bigr)_n}{\bigl(a+\frac{1}{2}\bigr)_n}
\intertext{and}
\label{16}
{\,}_3F_2\begin{bmatrix}\begin{gathered}-2n-1, a, d+1\\ 2a+1, d \end{gathered};2\end{bmatrix}
&=\frac{\bigl(1-\frac{2a}{d}\bigr)\bigl(\frac{3}{2}\bigr)_n}{\bigl(2a+1\bigr)\bigl(a+\frac{3}{2}\bigr)_n}
\end{align}
for $n\in\mathbb{N}_0$.
When $d=2a$, the extensions~\eqref{15} and~\eqref{16} reduce to~\eqref{6} and~\eqref{7}, respectively. 
\par
In 2021, Awad et al.~\cite{3} established one more extension of Kummer's second theorem~\eqref{5} by
\begin{multline}\label{17}
\te^{-x/2} {\,}_2F_2\begin{bmatrix}\begin{gathered}a,d+1\\ 2a+2,d\end{gathered};x\end{bmatrix}
={\,}_0F_1\begin{bmatrix}\begin{gathered}\\ a+\tfrac{3}{2}\end{gathered};\dfrac{x^2}{16}\end{bmatrix}
+\frac{x\bigl(\frac{a}{d}-1\bigr)}{2(a+1)}{\,}_0F_1\begin{bmatrix}\begin{gathered}\\ a+\tfrac{3}{2}\end{gathered};\dfrac{x^2}{16}\end{bmatrix}\\
+\frac{x^2\bigl(1-\frac{a}{d}\bigr)}{4(a+1)(2a+3)}{\,}_0F_1\begin{bmatrix}\begin{gathered}\\ a+\tfrac{5}{2}\end{gathered};\dfrac{x^2}{16}\end{bmatrix}.
\end{multline}
From~\eqref{17}, they further acquired
\begin{align}\label{18}
{\,}_3F_2\begin{bmatrix}\begin{gathered}-2n, a, d+1\\ 2a+2, d \end{gathered};2\end{bmatrix}
&=\frac{\bigl(\frac{1}{2}\bigr)_n}{\bigl(a+\frac{3}{2}\bigr)_n}\biggl[1+\frac{2n\bigl(1-\frac{a}{d}\bigr)}{a+1}\biggr]
\intertext{and}
\label{19}
{\,}_3F_2\begin{bmatrix}\begin{gathered}-2n-1, a, d+1\\ 2a+2, d \end{gathered};2\end{bmatrix}
&=\biggl(1-\frac{a}{d}\biggr)\frac{\bigl(\frac{3}{2}\bigr)_n}{(a+1)\bigl(a+\frac{3}{2}\bigr)_n}
\end{align}
for $n\in\mathbb{N}_0$.
When $d=a$, the equalities~\eqref{18} and~\eqref{19} reduce to~\eqref{6} and~\eqref{7}, respectively.

\subsection{Four classical combinatorial identities}
In this section, we review four classical combinatorial identities and outline their historical development.

\subsubsection{Knuth's old sum}
In terms of binomial coefficients $\binom{n}{k}$ and the Pochhammer symbol $(z)_n$, the well‑known combinatorial sum commonly referred to as Knuth's old sum~\cite{28}, or equivalently Reed Dawson's combinatorial identity, is given for $\nu\in\mathbb{N}_0$ by
\begin{equation}\label{20}
\sum_{k=0}^{n}\frac{(-1)^k}{2^k}\binom{n}{k}\binom{2k}{k}
=\begin{dcases}
\frac{1}{2^{2\nu}}\binom{2\nu}{\nu}=\frac{\bigl(\frac12\bigr)_\nu}{(1)_\nu}, &n=2\nu;\\
0; &n=2\nu+1.
\end{dcases}
\end{equation}
\par
It is worth noting that Reed Dawson communicated the above identity~\eqref{20} privately to Riordan, who later recorded it in his well-known book~\cite[p.~71]{28}. A variety of interesting and distinct proofs of these sums have appeared in the literature on combinatorial identities; for a comprehensive overview, see the survey by Prodinger~\cite{21}. We also note that Janassen and Knuth~\cite{12} provided an elementary proof of~\eqref{20} based on a recursion for the binomial coefficients. Gessel~\cite{10} represented the binomial coefficients $\binom{n}{k}$ as coefficients in suitable generating functions, while Rousseau~\cite{10} showed that the sums~\eqref{20} can be expressed as the constant term in the expansion of $\bigl(x^2+\frac{1}{x^2}\bigr)^n$. Finally, Prodinger~\cite{20} employed the classical Euler transformation to obtain another proof of~\eqref{20}.
\par
In 1974, Andrews~\cite[p.~478]{1} established the above sums~\eqref{20} by employing Gauss' second summation theorem~\eqref{3}. In 2004, Choi et al.~\cite{5} utilized the results~\eqref{6} and~\eqref{7} to establish~\eqref{20}.
\par
For a natural generalization of Knuth's old sum~\eqref{20}, we refer to the paper~\cite{24}.

\subsubsection{Riordan's combinatorial identity}
The second combinatorial identity we are discussing is
\begin{equation}\label{2.2}
\sum_{k=0}^{n}\frac{(-1)^k}{2^k}\binom{n+1}{k+1}\binom{2k}{k}
=\begin{dcases}
\frac{\bigl(\frac32\bigr)_\nu}{(1)_\nu}, &n=2\nu\\
\frac{\bigl(\frac32\bigr)_\nu}{(1)_\nu}, &n=2\nu+1
\end{dcases}
\end{equation}
for $\nu\in\mathbb{N}_0$. Riordan~\cite{28} established the combinatorial identity~\eqref{2.2} by the method of inverse relations.
\par
Lim~\cite{19} established a generalization of the above identity~\eqref{2.2} by employing the results~\eqref{9} and~\eqref{10}. 
\par
For another extension of Knuth's old sum~\eqref{20} and Riordan's combinatorial identity~\eqref{2.2}, we refer to the paper~\cite{13}.

\subsubsection{Gould's combinatorial identity}
For $n\in\mathbb{N}_0$, the combinatorial identity
\begin{equation}\label{22}
\sum_{k=0}^{\lfloor n/2\rfloor}\frac{1}{2^{2k}}\binom{n}{2k}\binom{2k}{k} =\frac{1}{2^n}\binom{2n}{n}=2^{n}\frac{\bigl(\frac{1}{2}\bigr)_n}{(1)_n}
\end{equation}
is valid, where $\lfloor x\rfloor$ stands for the floor function whose value equals the largest integer less than or equal to $x$; see~\cite[p.~34, Eq.~(3.99)]{8}.

\subsubsection{Touchard's combinatorial identity}
For $n\in\mathbb{N}_0$, the combinatorial identity
\begin{equation}\label{23}
\sum_{k=0}^{\lfloor n/2\rfloor}\binom{n}{2k}\frac{C_k}{2^{2k}}=\frac{C_{n+1}}{2^n}=2^{n}\frac{\bigl(\frac{3}{2}\bigr)_n}{(3)_n}
\end{equation}
is valid, see~\cite{33}, where
\begin{equation}\label{Catalan}
C_n=\frac{1}{n+1}\binom{2n}{n}, \quad n\in\mathbb{N}_0
\end{equation}
stands for the Catalan numbers; see~\cite{18, Catalan-Int-Surv.tex}.
In~\cite{11, 29, 30}, there have been several different proofs of Touchard's combinatorial identity~\eqref{23}, which are combinatorial in nature. In the paper~\cite{31}, Shpiro gave a combinatorial interpretation of the number $\binom{n}{2k}2^{n-2k}C_k$.
\par
In~\cite{27}, Rathie and Lim generalized the identity~\eqref{23} and mentioned that the identities~\eqref{2.2} and~\eqref{22} can be established with the help of Gauss' first summation theorem~\eqref{2}.
\par
It is not difficult to see that, in terms of the Pochhammer symbol $(z)_n$, the binomial numbers $\binom{n}{k}$ and the Catalan numbers~\eqref{Catalan} can be reformulated as
\begin{align}\label{25}
\binom{2n}{n}&=2^{2n}\frac{\bigl(\frac{1}{2}\bigr)_n}{(1)_n}, & \binom{n}{k}&=\frac{(-1)^k(-n)_k}{(1)_k},\\
\label{27}
\binom{k+n}{k}&=\frac{\bigl(n+1\bigr)_k}{(1)_k}, & \binom{n+1}{k+1}&=(-1)^k(n+1)\frac{\bigl(-n\bigr)_k}{(2)_k},\\
\label{29}
\binom{k+d}{k+1}&=\frac{d(d+1)_k}{(2)_k}, & \binom{n}{2k}&=\frac{\bigl(-\frac{n}2\bigr)_k\bigl(-\frac{n}2+\frac12\bigr)_k}{\bigl(\frac12\bigr)_k(1)_k},\\
\label{26}
\binom{k+d}{k}&=\frac{(d+1)_k}{(1)_k}, & C_n&=2^{2n}\frac{\bigl(\frac{1}{2}\bigr)_n}{(2)_n}.
\end{align}
These transforms will be used in Sections~\ref{sec-four} and~\ref{section5} below.
\par
It is well known that binomial coefficients and their reciprocals play an important role in many areas of mathematics, including number theory, probability, and statistics. In particular, sums involving central binomial coefficients have been studied extensively over a long period. Numerous properties of central binomial coefficients and their reciprocals can be found in the monograph~\cite{18} and the papers~\cite{Catalan-Int-Surv.tex, ScienceAsia-2022-0169.tex, arcsin-power-wei.tex}. Gould~\cite{8} has compiled a large collection of combinatorial identities involving central binomial coefficients, and the book by Riordan~\cite{28} also serves as a valuable reference.
\par
The applications of the Gauss hypergeometric function ${\,}_2F_1$ and the generalized hypergeometric functions ${}_pF_q$ in applied mathematics, number theory, probability, statistics, engineering mathematics, and combinatorial analysis are well established. In particular, the use of hypergeometric functions for evaluating binomial sums was originally suggested by Andrews~\cite{1}.
In this approach, the given binomial sum is first expressed in terms of an ordinary or generalized hypergeometric function. This is achieved by absorbing any polynomial terms in the summation index into the binomial coefficients, expanding the binomials in terms of factorials, and then converting those factorials into the Pochhammer symbols. Once the sum has been successfully transformed into a hypergeometric function, it is identified with known summation theorems available in the literature. A closed-form evaluation is then obtained when a suitable identity is found. This technique will be illustrated in Sections~\ref{sec-kummer} and~\ref{sec-four}.
\par
In the present work, we focus on Kummer's second theorem~\eqref{5} and four well-known combinatorial identities---Knuth's old sum~\eqref{20} (also known as Reed Dawson's combinatorial identity), Riordan's combinatorial identity~\eqref{2.2}, Gould's combinatorial identity~\eqref{22}, and Touchard's combinatorial identity~\eqref{23}---which we will discuss in turn.
\par
In this paper, we aim to establish extensions of the above-mentioned Kummer's second theorem~\eqref{5} and four famous combinatorial identities. For this, the rest of this paper is organized as follows.
\par
In Section~\ref{sec-kummer}, we further extend Kummer's second theorem~\eqref{5}. In Section~\ref{sec3}, we establish an extension of Knuth's classical sum~\eqref{20} by applying the results in~\eqref{15} and~\eqref{16}. Section~\ref{sec4} presents an extension of Riordan's combinatorial identity~\eqref{2.2}, derived using~\eqref{18} and~\eqref{19}. In Sections~\ref{sec5} and~\ref{sec6}, we develop extensions of Gould's identity~\eqref{22} and Touchard's identity~\eqref{23}, respectively; both are obtained with the aid of the extended form of Gauss's summation theorem~\eqref{4}. In Section~\ref{section5}, additional extensions of these four combinatorial identities are presented in the form of two theorems. Each section concludes with several known results and new identities that arise as special cases. The results presented in this paper are concise, of independent interest, and readily derived, and they may prove useful in further investigations.

\section{Extension of Kummer's second theorem}\label{sec-kummer}
In this section, we will establish a natural extension of the well-known Kummer's second theorem~\eqref{5}.

\begin{thm}\label{Kummer-extension-thm}
For $2a,d\not\in\mathbb{Z}_0^-$, the equality
\begin{equation}\label{Kummer-extension-Eq}
\te^{-x/2}{\,}_2F_2\begin{bmatrix}
\begin{gathered}
a,d+1\\ 2a, d
\end{gathered};x
\end{bmatrix}
={\,}_1F_2\begin{bmatrix}
\begin{gathered}
1+\tfrac{d}{2}\\ a+\tfrac{1}{2},\tfrac{d}{2}
\end{gathered};\dfrac{x^2}{16}
\end{bmatrix}
+\frac{x}{2d}{\,}_0F_1\begin{bmatrix}
\begin{gathered}
\\ a+\tfrac{1}{2}
\end{gathered};\dfrac{x^2}{16}
\end{bmatrix}
\end{equation}
is valid.
\end{thm}

\begin{proof}
In order to establish the required identity, we first establish the following result
\begin{equation}\label{a-rathir}
{\,}_2F_2\begin{bmatrix}
\begin{gathered}
a,d+1\\2a,d
\end{gathered};x
\end{bmatrix}
={\,}_1F_1\begin{bmatrix}
\begin{gathered}
a\\2a
\end{gathered};x
\end{bmatrix}
+\frac{x}{2d}{\,}_1F_1\begin{bmatrix}
\begin{gathered}
a+1\\2a+1
\end{gathered};x
\end{bmatrix}.
\end{equation}
For this, starting with the left-hand side of~\eqref{a-rathir}, we have, upon using its definition
\begin{equation*}
S={\,}_2F_2\begin{bmatrix}
\begin{gathered}
a,d+1\\2a,d
\end{gathered};x
\end{bmatrix}
=\sum_{n=0}^{\infty}\frac{(a)_n(d+1)_n}{(2a)_n(d)_n}\frac{x^n}{n!}
\end{equation*}
and the result $\tfrac{(d+1)_n}{(d)_n}=1+\tfrac{n}{d}$, we have
\begin{equation*}
S=\sum_{n=0}^{\infty}\biggl(1+\frac{n}{d}\biggr)\frac{(a)_n}{(2a)_n}\frac{x^n}{n!}.
\end{equation*}
Splitting into two series and summing up the first series, we have
\begin{equation*}
S={\,}_1F_1\begin{bmatrix}
\begin{gathered}
a\\2a
\end{gathered};x
\end{bmatrix}
+\frac{1}{d}\sum_{n=1}^{\infty}\frac{(a)_n}{(2a)_n}\frac{x^n}{(n-1)!}.
\end{equation*}
Now replacing $n-1$ by $n$ and using the result $(a)_{n+1}=a(a+1)_n$, we have
\begin{equation*}
S={\,}_1F_1\begin{bmatrix}
\begin{gathered}
a\\2a
\end{gathered};x
\end{bmatrix}
+\frac{x}{2d}\sum_{n=0}^{\infty}\frac{(a+1)_n}{(2a+1)_n}\frac{x^n}{n!}.
\end{equation*}
Finally, summing up the second series, we easily arrive at the right-hand side of~\eqref{a-rathir}.
\par
Now we are ready to establish our identity~\eqref{Kummer-extension-Eq}. For this, multiplying both sides of~\eqref{a-rathir} by $\te^{-x/2}$, we have
\begin{equation*}
\te^{-x/2}{\,}_2F_2\begin{bmatrix}
\begin{gathered}
a,d+1\\2a,d
\end{gathered};x
\end{bmatrix}=\te^{-x/2}{\,}_1F_1\begin{bmatrix}
\begin{gathered}
a\\ 2a
\end{gathered}; x
\end{bmatrix}
+\frac{x}{2d}\te^{-x/2}{\,}_1F_1\begin{bmatrix}\begin{gathered}
a+1\\2a+1
\end{gathered};x\end{bmatrix}.
\end{equation*}
Now observe that the first expression on the right-hand side can be evaluated with the help of the well-known Kummer's second theorem~\eqref{5} and the second expression can be evaluated with the help of the contiguous Kummer's second theorem~\eqref{1.7a}.
\par
Direct computation gives
\begin{multline}\label{4-rathie}
\te^{-x/2}{\,}_2F_2\begin{bmatrix}
\begin{gathered}
a,d+1\\2a,d
\end{gathered}; x
\end{bmatrix}\\
\begin{aligned}
&={\,}_0F_1\begin{bmatrix}
\begin{gathered}
\\a+\tfrac12
\end{gathered};\dfrac{x^2}{16}\end{bmatrix}
+\frac{x}{2d}\biggl({\,}_0F_1\begin{bmatrix}
\begin{gathered}
\\a+\tfrac12
\end{gathered};\dfrac{x^2}{16}\end{bmatrix}
+\frac{x}{2(2a+1)}{\,}_0F_1\begin{bmatrix}
\begin{gathered}
\\a+\tfrac32
\end{gathered};\dfrac{x^2}{16}
\end{bmatrix}\biggr)\\
&={\,}_0F_1\begin{bmatrix}
\begin{gathered}
\\a+\tfrac12
\end{gathered};\dfrac{x^2}{16}\end{bmatrix}
+\frac{x^2}{4d(2a+1)}{\,}_0F_1\begin{bmatrix}
\begin{gathered}
\\a+\tfrac32
\end{gathered};\dfrac{x^2}{16}\end{bmatrix}
+\frac{x}{2d}{\,}_0F_1\begin{bmatrix}
\begin{gathered}
\\a+\tfrac12
\end{gathered};\dfrac{x^2}{16}
\end{bmatrix}.
\end{aligned}
\end{multline}
But it is not difficult to prove that
\begin{equation*}
{\,}_0F_1\begin{bmatrix}
 \begin{gathered}
 \\ a+\tfrac12
 \end{gathered};\dfrac{x^2}{16}
\end{bmatrix}
+\frac{x^2}{4d(2a+1)}{\,}_0F_1\begin{bmatrix}
 \begin{gathered}
 \\ a+\tfrac32
 \end{gathered};\dfrac{x^2}{16}
\end{bmatrix}
={\,}_1F_2\begin{bmatrix}
 \begin{gathered}
 1+\tfrac{d}{2}\\ a+\tfrac12,\tfrac{d}{2}
 \end{gathered};\dfrac{x^2}{16}
\end{bmatrix}.
\end{equation*}
Hence, in view of~\eqref{4-rathie}, we finally obtain
\begin{equation*}
\te^{-x/2}{\,}_2F_2\begin{bmatrix}
 \begin{gathered}a,d+1\\ 2a, d\end{gathered};x
\end{bmatrix}
={\,}_1F_2\begin{bmatrix}
 \begin{gathered}
 1+\tfrac{d}{2}\\a+\tfrac{1}{2},\tfrac{d}{2}
 \end{gathered};\dfrac{x^2}{16}
\end{bmatrix}
+\frac{x}{2d}{\,}_0F_1\begin{bmatrix}
 \begin{gathered}
 \\a+\tfrac{1}{2}
 \end{gathered};\dfrac{x^2}{16}
\end{bmatrix}.
\end{equation*}
This completes the proof of~\eqref{Kummer-extension-Eq}. The proof of Theorem~\ref{Kummer-extension-thm} is complete.
\end{proof}

\begin{rem}
Letting $d\to\infty$ in~\eqref{Kummer-extension-Eq} readily leads to Kummer's second theorem~\eqref{5}. The identity~\eqref{Kummer-extension-Eq} is closely related to~\eqref{14} and~\eqref{17}.
\end{rem}

\begin{rem}
When taking $d=2a-1$ in~\eqref{Kummer-extension-Eq}, we immediately recover the known result~\eqref{1.7a}.
\end{rem}

\section{Two new formulas for the terminating series ${\,}_3F_2$}
In this section, we will establish two new interesting results for the terminating series ${\,}_3F_2$.

\begin{thm}
For $2a,d\not\in\mathbb{Z}_0^-$ and $n\in\mathbb{N}_0$, we have
\begin{align}\label{B.1}
{\,}_3F_2\begin{bmatrix}
\begin{gathered}
-2n,a,d+1\\ 2a,d
\end{gathered};2
\end{bmatrix}
&=\frac{\bigl(\frac{1}{2}\bigr)_n}{\bigl(a+\frac{1}{2}\bigr)_n}\frac{\bigl(1+\frac{d}{2}\bigr)_n}{\bigl(\frac{d}{2}\bigr)_n}
\intertext{and}
\label{B.2}
{\,}_3F_2\begin{bmatrix}
\begin{gathered}
-2n-1,a,d+1\\ 2a,d
\end{gathered};2
\end{bmatrix}
&=-\frac{1}{d}\frac{\bigl(\frac{3}{2}\bigr)_n}{\bigl(a+\frac{1}{2}\bigr)_n}.
\end{align}
\end{thm}

\begin{proof}
Denote the left-hand side of~\eqref{Kummer-extension-Eq} by $S_1$. Expanding two functions into series yields
\begin{equation*}
S_1=\sum_{n=0}^{\infty}\sum_{m=0}^{\infty}\frac{(-1)^n}{2^nn!}\frac{(a)_m(d+1)_m}{(2a)_m(d)_m}\frac{x^{m+n}}{m!n!}.
\end{equation*}
Replacing $n$ by $n-m$ and using
\begin{equation*}
\sum_{n=0}^{\infty}\sum_{k=0}^{\infty}A(k,n)=\sum_{n=0}^{\infty}\sum_{k=0}^{n}A(k,n-k)
\end{equation*}
in~\cite[p.~56, Lemma~10]{23}, we have
\begin{equation*}
S_1=\sum_{n=0}^{\infty}\sum_{m=0}^{n}\frac{(-1)^{n-m}}{2^{n-m}(n-m)!} \frac{(a)_m(d+1)_m}{(2a)_m(d)_m}\frac{x^n}{m!}.
\end{equation*}
In view of the identity $(n-m)!=(-1)^m\frac{n!}{(-n)_m}$, we obtain
\begin{equation*}
S_1=\sum_{n=0}^{\infty}\frac{(-1)^n}{2^n}\frac{x^n}{n!} \sum_{m=0}^{n}\frac{(-n)_m(a)_m(d+1)_m}{(2a)_m(d)_m}\frac{2^m}{m!}.
\end{equation*}
Summing up the inner series, we acquire
\begin{equation*}
S_1=\sum_{n=0}^{\infty}\frac{(-1)^n}{2^n}\frac{x^n}{n!}{\,}_2F_2\begin{bmatrix}
\begin{gathered}
-n,a,d+1\\2a,d
\end{gathered};2
\end{bmatrix}.
\end{equation*}
Separating the last series into even and odd powers of $x$ and using the elementary identities
\begin{equation*}
(2n)!=2^{2n}n!\biggl(\frac{1}{2}\biggr)_n \quad\text{and}\quad (2n+1)!=2^{2n}n!\biggl(\frac{3}{2}\biggr)_n,
\end{equation*}
we have
\begin{multline}\label{B.3}
S_1=\sum_{n=0}^{\infty}\frac{1}{2^{4n}\bigl(\frac{1}{2}\bigr)_n}\frac{x^{2n}}{n!} {\,}_3F_2\begin{bmatrix}
\begin{gathered}
-2n,a,d+1\\2a,d
\end{gathered};2
\end{bmatrix}\\
-\sum_{n=0}^{\infty}\frac{1}{2^{4n+1}\bigl(\frac{3}{2}\bigr)_n}\frac{x^{2n+1}}{n!} {\,}_3F_2\begin{bmatrix}
\begin{gathered}
-2n-1,a,d+1\\2a,d
\end{gathered};2
\end{bmatrix}.
\end{multline}
\par
On the other hand, if we denote the right-hand side of~\eqref{Kummer-extension-Eq} by $S_2$, then
\begin{multline}\label{B.4}
S_2={\,}_1F_2\begin{bmatrix}
\begin{gathered}
1+\tfrac{d}{2}\\a+\tfrac{1}{2}, \tfrac{d}{2}
\end{gathered};\dfrac{x^2}{16}
\end{bmatrix}
+\frac{x}{2d} {\,}_0F_1\begin{bmatrix}
\begin{gathered}
\\a+\tfrac{1}{2}
\end{gathered};\dfrac{x^2}{16}
\end{bmatrix}\\
=\sum_{n=0}^{\infty}\frac{\bigl(1+\frac{d}{2}\bigr)_n} {\bigl(a+\frac{1}{2}\bigr)_n\bigl(\frac{d}{2}\bigr)_n}\frac{x^{2n}}{2^{4n}n!} +\frac{1}{d}\sum_{n=0}^{\infty}\frac{1}{2^{4n+1}\bigl(a+\frac{1}{2}\bigr)_n}\frac{x^{2n+1}}{n!}.
\end{multline}
Thus, from~\eqref{B.3} and~\eqref{B.4}, equating coefficients of $x^{2n}$ and $x^{2n+1}$, we easily arrive at the desired results~\eqref{B.1} and~\eqref{B.2}. The proof is thus complete.
\end{proof}

\begin{rem}
In~\eqref{B.1} and~\eqref{B.2}, if taking $d\to\infty$, we at once recover the classical results~\eqref{6} and~\eqref{7}, respectively.
\end{rem}

\begin{rem}
In~\eqref{B.1} and~\eqref{B.2}, if taking $d=a$ and then replacing $a$ by $a-1$, we obtain
\begin{align*}
{\,}_2F_1\begin{bmatrix}
\begin{gathered}
-2n,a\\2a-2
\end{gathered};2
\end{bmatrix}
&=\frac{\bigl(\frac{1}{2}\bigr)_n}{\bigl(a-\frac{1}{2}\bigr)_n}\frac{a-1+2n}{a-1}
\intertext{and}
{\,}_2F_1\begin{bmatrix}
\begin{gathered}
-2n-1,a\\2a-2
\end{gathered};2
\end{bmatrix}
&=-\frac{1}{a-1}\frac{\bigl(\frac{3}{2}\bigr)_n}{\bigl(a-\frac{1}{2}\bigr)_n}.
\end{align*}
These results are recorded in~\cite{16, 17}.
\end{rem}

\section{Extensions of four famous combinatorial identities}\label{sec-four}
We now start out to extend the identities~\eqref{20}, \eqref{2.2}, \eqref{22}, and~\eqref{23}.

\subsection{An extension of Knuth's old sum}\label{sec3}

In this section, we present two extensions of Knuth's old sum.

\begin{thm}\label{thm1}
For $d\neq 0,-1,-2,\dotsc$ and $n,\nu\in\mathbb{N}_0$, we have
\begin{equation}\label{Eqq3.5}
\sum_{k=0}^{n}\frac{(-1)^k}{2^k}\binom{n}{k}\binom{2k}{k}\frac{\binom{k+d}{k}}{\binom{k+d-1}{k}}
=\begin{dcases}
\frac{\bigl(\frac{1}{2}\bigr)_\nu}{(1)_\nu}\biggl(1+\frac{2\nu}{d}\biggr), & n=2\nu;\\
-\frac{1}{d}\frac{\bigl(\frac{3}{2}\bigr)_\nu}{(1)_\nu}, & n=2\nu+1.
\end{dcases}
\end{equation}
\end{thm}

\begin{proof}
Denote the left-hand side of~\eqref{Eqq3.5} by $S_1$.
Converting all binomial coefficients involved in $S_1$ to the Pochhammer symbols by employing the appropriate elementary identities between~\eqref{25} and~\eqref{26} results in
$$
S_1=\sum_{k=0}^{n}\frac{(-1)^k}{2^k}\binom{n}{k}\binom{2k}{k}\frac{\binom{k+d}{k}}{\binom{k+d-1}{k}}
={\,}_3F_2\begin{bmatrix}\begin{gathered}-n, \tfrac12, d+1\\ 1, d \end{gathered};2\end{bmatrix}.
$$
Taking $a=\frac{1}{2}$ in~\eqref{B.1} and~\eqref{B.2} immediately leads to 
\begin{equation*}
{\,}_3F_2\begin{bmatrix}\begin{gathered}-n, \tfrac12, d+1\\ 1, d \end{gathered};2\end{bmatrix}=\begin{dcases}
\frac{\bigl(\frac12\bigr)_\nu}{(1)_\nu}\biggl(1+\frac{2\nu}{d}\biggr), & n=2\nu;\\
-\frac{1}{d}\frac{\bigl(\frac32\bigr)_\nu}{(1)_\nu}, & n=2\nu+1.
\end{dcases}
\end{equation*}
Accordingly, the identity~\eqref{Eqq3.5} is proved. Theorem~\ref{thm1} is thus proved.
\end{proof}

\begin{rem}
When $d\to\infty$ in~\eqref{Eqq3.5}, we recover~\eqref{20}. When $d=1$ in~\eqref{Eqq3.5}, we arrive at
\begin{equation*}
\sum_{k=0}^{n}(-1)^k\frac{k+1}{2^k}\binom{n}{k}\binom{2k}{k}
=\begin{dcases}
\frac{\bigl(\frac{3}{2}\bigr)_\nu}{(1)_\nu}, & n=2\nu;\\
-\frac{\bigl(\frac{3}{2}\bigr)_\nu}{(1)_\nu}, & n=2\nu+1.
\end{dcases}
\end{equation*}
\end{rem}

\subsection{An extension of Riordan's combinatorial identity}\label{sec4}

In this section, we establish an extension of Riordan's combinatorial identity.

\begin{thm}\label{thm2}
For $d\neq 0,-1,-2,\dotsc$ and $n,\nu\in\mathbb{N}_0$, we have
\begin{multline}\label{33}
\sum_{k=0}^{n}\frac{(-1)^k}{2^k}\frac{\binom{n+1}{k+1}}{\binom{k+d}{k+1}}\frac{\binom{k+d+1}{k+1}}{\binom{k+2}{k+1}}\binom{2k}{k}\\
=\begin{dcases} \frac{(2\nu+1)(d+1)}{2d}\biggl[1+\frac{2\nu}{3} \biggl(\frac{2d+1}{d+1}\biggr)\biggr]\frac{\bigl(\frac12\bigr)_\nu}{(2)_\nu}, & n=2\nu;\\
\frac{(\nu+1)(2d+1)}{3d}\frac{\bigl(\frac32\bigr)_\nu}{(2)_\nu}, & n=2\nu+1.
\end{dcases}
\end{multline}
\end{thm}

\begin{proof}
Denote the left-hand side of~\eqref{33} by $S_2$.
Using the elementary identities in \eqref{27} and \eqref{29}, we obtain
\begin{align*}
S_2&=\frac{(n+1)(d+1)}{2d}\sum_{k=0}^{n}\frac{(-n)_k\bigl(\frac{1}{2}\bigr)_k\bigl(d+2\bigr)_k}{(3)_k(d+1)_k}\frac{2^k}{k!}\\
&=\frac{(n+1)(d+1)}{2d}{\,}_3F_2\begin{bmatrix}\begin{gathered}-n, \tfrac12, d+2\\ 3, d+1 \end{gathered};2\end{bmatrix}.
\end{align*}
Taking $a=\frac{1}{2}$ and replacing $d$ by $d+1$ in~\eqref{18} and~\eqref{19}, we obtain
\begin{equation*}
{\,}_3F_2\begin{bmatrix}\begin{gathered}-n, \tfrac12, d+2\\ 3, d+1 \end{gathered};2\end{bmatrix}
=\begin{dcases} \biggl(1+\frac{2\nu}{3}\frac{2d+1}{d+1}\biggr)\frac{\bigl(\frac12\bigr)_\nu}{(2)_\nu}, & n=2\nu;\\
\frac{2d+1}{3(d+1)}\frac{\bigl(\frac32\bigr)_\nu}{(2)_\nu}, & n=2\nu+1.
\end{dcases}
\end{equation*}
The identity~\eqref{33} is thus proved.
The proof of Theorem~\ref{thm2} is complete.
\end{proof}

\begin{rem}
Letting $d=1$ in the identity~\eqref{33}, we immediately recover Riordan's combinatorial identity~\eqref{2.2}. 
\end{rem}

\begin{rem}
Letting $d=2$ in the identity~\eqref{33} results in
\begin{equation*}
\sum_{k=0}^{n}\frac{(-1)^k}{2^k}\frac{k+3}{k+2}\binom{n+1}{k+1}{\binom{2k}{k}}
=\begin{dcases}
\frac{(2\nu+1)(10\nu+9)}{6}\frac{\bigl(\frac12\bigr)_\nu}{(2)_\nu}, & n=2\nu\\
\frac{5}{3}\frac{\bigl(\frac32\bigr)_\nu}{(1)_\nu}, &n=2\nu+1
\end{dcases}
\end{equation*}
for $n,\nu\in\mathbb{N}_0$.
\end{rem}

\subsection{An extension of Gould's combinatorial identity}\label{sec5}

In this section, we derive two extensions of Gould's combinatorial identity.

\begin{thm}\label{thm3}
For $d\neq 0,-1,-2,\dotsc$ and $n\in\mathbb{N}_0$, we have
\begin{equation}\label{Eq3.8}
\sum_{k=0}^{\lfloor n/2\rfloor}\frac{1}{2^{2k}}\binom{n}{2k}\binom{2k}{k}\frac{\binom{k+d}{k}}{\binom{k+d-1}{k}}
=\biggl[\frac{n(n-1)}{4d}+n-\frac{1}{2}\biggr]\frac{2^n\Gamma\bigl(n-\frac{1}{2}\bigr)}{\sqrt{\pi}\,\Gamma(n+1)}.
\end{equation}
\end{thm}

\begin{proof}
Denote the left-hand side of~\eqref{Eq3.8} by $S_3$. Converting all binomial coefficients involved in $S_3$ to the Pochhammer symbols by employing the appropriate elementary identities between~\eqref{25} and~\eqref{26} results in
$$
S_3=\sum_{k=0}^{\lfloor n/2\rfloor}\frac{1}{2^{2k}}\binom{n}{2k}\binom{2k}{k}\frac{\binom{k+d}{k}}{\binom{k+d-1}{k}}
={\,}_3F_2\begin{bmatrix}\begin{gathered}-\tfrac{n}{2}, -\tfrac{n}2+\tfrac12, d+1\\ 1, d \end{gathered};1\end{bmatrix}.
$$
Taking $a=-\frac{n}2$, $b=-\frac{n}2+\frac12$, and $c=0$ in~\eqref{4}, we readily arrive at
\begin{equation*}
{\,}_3F_2\begin{bmatrix}\begin{gathered}-\tfrac{n}{2}, -\tfrac{n}2+\tfrac12, d+1\\ 1, d \end{gathered};1\end{bmatrix}
=\biggl[\frac{n(n-1)}{4d}+n-\frac{1}{2}\biggr]\frac{2^n\Gamma\bigl(n-\frac{1}{2}\bigr)}{\sqrt{\pi}\,\Gamma(n+1)}.
\end{equation*}
The identity~\eqref{Eq3.8} is thus proved. The proof of Theorem~\ref{thm3} is complete.
\end{proof}

\begin{rem}
When $d\to\infty$ in~\eqref{Eq3.8}, we recover Gould's combinatorial identity~\eqref{22}. When $d=1$ in~\eqref{Eq3.8}, we acquire
\begin{equation*}
\sum_{k=0}^{\lfloor n/2\rfloor}\frac{k+1}{2^{2k}}\binom{n}{2k}\binom{2k}{k}
=2^{n-2}\bigl(n^2+3n-2\bigr)\frac{\Gamma\bigl(n-\frac{1}{2}\bigr)}{\sqrt{\pi}\,\Gamma(n+1)}.
\end{equation*}
\end{rem}

\subsection{An extension of Touchard's combinatorial identity}\label{sec6}

In this section, we establish an extension of Touchard's combinatorial identity.

\begin{thm}\label{thm4}
For $d\neq 0,-1,-2,\dotsc$ and $n\in\mathbb{N}_0$, we have
\begin{equation}\label{35}
\sum_{k=0}^{\lfloor n/2\rfloor}\frac{1}{2^{2k}}\frac{\binom{k+d+1}{k}}{\binom{k+d}{k}} \frac{\binom{2k}{k}}{\binom{k+2}{k}}\binom{n}{2k}
=\frac{2^{n+1}}{d+1}\frac{n^{2}+n(4d+3)+6(d+1)}{(n+3)(n+4)}\frac{(\frac{3}{2})_{n}}{(3)_{n}}.
\end{equation}
\end{thm}
\begin{proof}
Denote the left-hand side of~\eqref{35} by $S_4$. By virtue of the first one in~\eqref{29}, \eqref{25}, the first one in~\eqref{27}, and the first one in~\eqref{26}, converting all binomial coefficients involved in $S_4$ to the Pochhammer symbols yields
$$
S_4=\sum_{k=0}^{\lfloor n/2\rfloor} \frac{\bigl(-\frac{n}2\bigr)_k\bigl(-\frac{n}2+\frac12\bigr)_n\bigl(d+2\bigr)_k}{(3)_k(d+1)_k}\frac{2^k}{k!}
={\,}_3F_2\begin{bmatrix}\begin{gathered}-\tfrac{n}{2}, -\tfrac{n}2+\tfrac12, d+2\\ 3, d+1 \end{gathered};1\end{bmatrix}.
$$
Further taking $a=-\frac{n}2$, $b=-\frac{n}2+\frac12$, and $c=2$ in the extended Gauss summation theorem~\eqref{4} leads to the identity~\eqref{35}.
The proof of Theorem~\ref{thm4} is complete.
\end{proof}

\begin{rem}
Letting $d=1$ in the identity~\eqref{35}, we recover Touchard's combinatorial identity~\eqref{23} readily.
\end{rem}

We conclude this section by noting that one may ask whether there exists a single identity that encompasses Knuth's old sum (also known as Reed--Dawson's identity) and Riordan's combinatorial identities, and likewise includes Gould's identity and Touchard's identity as special cases. The next section provides an affirmative answer.

\section{Further extensions of four famous combinatorial identities}\label{section5}

In this section, we now start out to extend the identities~\eqref{20} to~\eqref{23} in the following two theorems.

\subsection{Extension of Knuth's old sum and Riordan's combinatorial identity}

In this section, we present an interesting combinatorial identity that includes Knuth's old sum and Riordan's combinatorial identity.

\begin{thm}\label{5.1thm}
For $d \neq 0,-1,-2,\dotsc$ and $n,\nu \in \mathbb{N}_0$, we have
\begin{equation}\label{(5.1)}
\sum_{k=0}^{n}\frac{(-1)^k}{2^k(k+1)}
\binom{n}{k}\binom{2k}{k}
\frac{\binom{k+d}{k}}{\binom{k+d-1}{k}}
=
\begin{dcases}
\frac{\bigl(\frac12\bigr)_\nu}{(1)_\nu}, & n=2\nu; \\
\frac12\biggl(1-\frac{1}{d}\biggr)
\frac{\bigl(\frac32\bigr)_\nu}{(2)_\nu}, & n=2\nu+1.
\end{dcases}
\end{equation}
\end{thm}

\begin{proof}
Denote the left-hand side of~\eqref{(5.1)} by $S_1$. Converting all binomial coefficients involved in $S_1$ to the Pochhammer symbols by employing the appropriate elementary identities~\eqref{25} to~\eqref{26} results in
\[
S_1
= \sum_{k=0}^{n}
\frac{(-n)_k\bigl(\frac12\bigr)_k(d+1)_k}{(2)_k(d)_k}\frac{2^k}{k!}
= {}_3F_2
\begin{bmatrix}
\begin{gathered}-n,\tfrac12,d+1 \\
2,d\end{gathered}; 2
\end{bmatrix}.
\]
Taking $a=\frac12$ in~\eqref{15} and~\eqref{16} immediately leads to
\[
{}_3F_2\begin{bmatrix}\begin{gathered}-n, \tfrac{1}{2}, d+1\\
2,d\end{gathered};2\end{bmatrix}
=
\begin{dcases}
\frac{\bigl(\frac12\bigr)_\nu}{(1)_\nu}, & n=2\nu; \\
\frac12\biggl(1-\frac{1}{d}\biggr)
\frac{\bigl(\frac32\bigr)_\nu}{(2)_\nu}, & n=2\nu+1.
\end{dcases}
\]
Accordingly, the identity~\eqref{(5.1)} is proved. Theorem~\ref{5.1thm} is thus proved.
\end{proof}

\begin{rem}
It is interesting that the identity~\eqref{(5.1)} in Theorem~\ref{5.1thm} for the case \( n=2\nu \) is independent of \( d \).
\end{rem}

\begin{rem}
For \( d=1 \) in~\eqref{(5.1)}, we immediately recover Knuth's old sum~\eqref{20}, while taking the limit \( d \to \infty \) in~\eqref{(5.1)} yields Riordan's identity~\eqref{2.2}.
\end{rem}

\begin{rem}
Letting \( d=\tfrac{1}{2} \) in the identity~\eqref{(5.1)}, we acquire
\[
\sum_{k=0}^{n} (-1)^k\frac{2k+1}{2^k (k+1)}\binom{n}{k}\binom{2k}{k}
=
\begin{dcases}
\frac{\bigl(\frac12\bigr)_\nu}{(1)_\nu}, & n=2\nu; \\
-\frac{1}{2}\frac{\bigl(\frac32\bigr)_\nu}{(2)_\nu}, & n=2\nu+1.
\end{dcases}
\]
\end{rem}

\subsection{Extension of Gould's identity and Touchard's identity}

In this section, we present an interesting combinatorial identity that
includes Gould's identity and Touchard's identity.

\begin{thm}\label{5.2thm}
For $d \neq 0,-1,-2,\dotsc$ and $n\in\mathbb{N}_0$, we have
\begin{equation}\label{5.2}
\sum_{k=0}^{n}\frac{1}{2^{2k}(k+1)}\binom{n}{2k}\binom{2k}{k}\frac{\binom{k+d}{k}}{\binom{k+d-1}{k}}
=\frac{1}{d}\frac{n(n-1)+2d(2n+1)}{2^{n}(n+1)(n+2)}\binom{2n}{n}.
\end{equation}
\end{thm}

\begin{proof}
Denote the left-hand side of~\eqref{5.2} by $S_2$. Converting all binomial coefficients involved in $S_2$ to the Pochhammer symbols by employing the appropriate elementary identities~\eqref{25} to~\eqref{26} results in
\[
S_2
= \sum_{k=0}^{\lfloor n/2 \rfloor}
\frac{\bigl(-\tfrac{n}{2}\bigr)_k\bigl(-\tfrac{n}{2}+\tfrac12\bigr)_k(d+1)_k}{(2)_k(d)_k}\frac{1}{k!}
= {}_3F_2\begin{bmatrix}
\begin{gathered}
-\tfrac{n}{2},-\tfrac{n}{2}+\tfrac12,d+1\\ 2,d
\end{gathered};1
\end{bmatrix}.
\]
Taking \(a=-\frac{n}{2}\), \(b=-\frac{n}{2}+\frac12\), and \(c=1\) in extended Gauss summation theorem~\eqref{4}, we readily arrive at
\[
{}_3F_2\begin{bmatrix}
\begin{gathered}
-\tfrac{n}{2},-\tfrac{n}{2}+\tfrac12,d+1\\ 2,d
\end{gathered};1
\end{bmatrix}=
\frac{1}{d}\frac{n(n-1)+2d(2n+1)}{2^{n}(n+1)(n+2)}\binom{2n}{n}.
\]
The identity~\eqref{5.2} is thus proved. The proof of Theorem~\ref{5.2thm} is complete.
\end{proof}

\begin{rem}
Setting $d = 1$ in the identity~\eqref{5.2} immediately recovers Gould's identity~\eqref{22}, while taking the limit $d \to \infty$ in the identity~\eqref{5.2} yields Touchard's identity~\eqref{23}.
\end{rem}

\section{Conclusions}

In this paper, we extended Kummer's second theorem~\eqref{5} and the following four well-known combinatorial identities:
\begin{enumerate}
\item 
Knuth's old sum~\eqref{20},
\item 
Riordan's combinatorial identity~\eqref{2.2},
\item 
Gould's combinatorial identity~\eqref{22},
\item 
Touchard's combinatorial identity~\eqref{23}.
\end{enumerate}
These extensions are established by employing the extensions of Gauss's summation theorem~\eqref{4} and the known results~\eqref{15}, \eqref{16}, \eqref{18}, and~\eqref{19} for terminating ${}_3F_{2}$ series. Owing to the presence of the general parameter $d$, the results presented in this paper are of a very general character. Consequently, they serve as master formulas from which a large number of new and interesting results, including many known special cases, can be derived. We believe that the identities established in this paper are new and do not appear in the existing literature, thereby constituting a genuine contribution to the theory of generalized hypergeometric identities. It is hoped that these results will prove useful in various areas of applied mathematics.
\par
We conclude this paper by noting that the natural extensions of the aforementioned generalizations of Knuth's old sum, Riordan's identity, Gould's identity, and Touchard's identity are currently under investigation and will be presented in a subsequent paper.
\par
By the way, in the papers~\cite{Gauss-2Formula.tex, DA19034-cas-sc.tex}, with the help of the Fa\`a di Bruno formula and several identities of partial Bell polynomials, by establishing and comparing different forms for explicit formulas of the Gauss hypergeometric functions
\begin{gather*}
{\,}_2F_1\biggl(\frac{1-n}{2},\frac{2-n}{2};\frac{3}{2}-m;z^2\biggr), \quad
{\,}_2F_1\biggl(-\frac{n}{2},\frac{1-n}{2};\frac{1}{2}-m;z^2\biggr),\\
{\,}_2F_1\biggl(a,a+\frac{1}{2};\frac{3}{2}-m;z^2\biggr), \quad
{\,}_2F_1\biggl(a,a+\frac{1}{2};\frac{1}{2}-m;z^2\biggr)
\end{gather*}
for $m,n\in\mathbb{N}$ and $a\in\mathbb{C}$, the authors discovered the following new combinatorial identities
\begin{align*}
\sum_{k=0}^{n} \frac{2^{k}}{k!}\binom{2n-2k}{n-k}
\sum_{j=0}^{k}\frac{(-1)^{j}}{2^j}
\frac{(2k-2j-1)!!}{(n-j)!}
\binom{2k-j-1}{j-1}
&=\frac{1}{n!},\\
\sum_{k=0}^{m}\frac{2^k}{k!} \binom{2m-2k}{m-k}
\sum_{\ell=0}^{k} \frac{(-1)^\ell}{2^\ell}
\frac{(2k-2\ell-1)!!}{(n-\ell)!}
\binom{2k-\ell-1}{\ell-1}
&=\frac{1}{n!}\binom{2m-n}{m},\\
\intertext{and}
\sum_{k=1}^{m}\frac{1}{(k!)^2}\binom{2m-2k}{m-k}
\sum_{\ell=1}^{k} \binom{k}{\ell}\ell(2k-\ell-1)! (2a)_\ell
&=\binom{2m+2a}{m},
\end{align*}
where $m\in\mathbb{N}_0$, $n\in\mathbb{Z}$, and $a\in\mathbb{C}$.
\par
This paper is a slightly modified version of the arXiv preprint at \url{https://arxiv.org/abs/2608.05192v1}.

\section{Declarations}

\paragraph{\bf Authors' Contributions}
All authors contributed equally to the manuscript and read and approved the final manuscript.

\paragraph{\bf Funding}
The second author was partially supported by the Natural Science Foundation of Inner Mongolia Autonomous Region (Grant No.~2025QN01041) and by the Youth Project of Hulunbuir City for Basic Research and Applied Basic Research (Grant No.~GH2024020).

\paragraph{\bf Institutional Review Board Statement}
Not applicable.

\paragraph{\bf Informed Consent Statement}
Not applicable.

\paragraph{\bf Ethical Approval}
The conducted research is not related to either human or animal use.

\paragraph{\bf Availability of Data and Material}
Data sharing is not applicable to this article as no new data were created or analyzed in this study.

\paragraph{\bf Competing Interests}
The authors declare that they have no any conflict of competing interests.

\paragraph{\bf Use of AI Tools Declaration}
The authors declare they have not used Artificial Intelligence (AI) tools in the creation of this article.

\paragraph{\bf Acknowledgements}
Not applicable.


\begin{thebibliography}{99}

\bibitem{1}
G. E. Andrews, \textit{Applications of basic hypergeometric functions}, SIAM Rev. \textbf{16} (1974), 441\nobreakdash--484. DOI: \url{https://doi.org/10.1137/1016081}.

\bibitem{3}
M. M. Awad, A. O. Mohammed, M. A. Rakha, and A. K. Rathie, \textit{On an interesting extension of Kummer's second theorem with applications}, Commun. Korean Math. Soc. \textbf{36} (2021), no.~1, 63\nobreakdash--101. DOI: \url{https://doi.org/10.4134/CKMS.c200147}.

\bibitem{5}
J. Choi, A. K. Rathie, and H. V. Harsh, \textit{A note on Reed Dawson identities}, Korean J. Math. Sci. \textbf{11} (2004), no.~2, 1\nobreakdash--4.

\bibitem{7}
C. F. Gauss, \textit{Disquisitiones generales circa seriem infinitam}, Gottingen, published in Ges Werke Gottingen (1866), Vol.~II, 437\nobreakdash--445; III 123\nobreakdash--163; III 207\nobreakdash--229; III 446\nobreakdash--460.

\bibitem{8}
H. W. Gould, \emph{Combinatorial Identities}, A standardized set of tables listing 500 binomial coefficient summations, Morgantown, WV, 1972.

\bibitem{10}
D. H. Greene, and D. E. Knuth, \textit{Mathematics for the Analysis of Algorithms}, Reprint of the third (1990) edition. Modern Birkh\"auser Classics. Birkh\"auser Boston, Inc., Boston, MA, 2008.

\bibitem{11}
H. Izbicki, \textit{\"Uber Unterb\"aume eines Baumes}, Monatsh. Math. \textbf{74} (1970), 56\nobreakdash--62. DOI: \url{https://doi.org/10.1007/BF01298302}. (German)

\bibitem{12}
A. T. Jonassen and D. E. Knuth, \textit{A trivial algorithm whose analysis isn't}, J. Comput. System Sci. \textbf{16} (1978), no.~3, 301\nobreakdash--322. DOI: \url{https://doi.org/10.1016/0022-0000(78)90020-X}.

\bibitem{13}
I. Kim, G. V. Milovanovi\'c, R. B. Paris, and A. K. Rathie, \textit{A note on a further extension of Gauss's second summation theorem with an application to the extension of two well-known combinatorial identities}, Quaest. Math. \textbf{45} (2022), no.~6, 959\nobreakdash--968. DOI: \url{https://doi.org/10.2989/16073606.2021.1925368}.

\bibitem{14}
Y. S. Kim, J. Choi, and A. K. Rathie, \textit{Two results for the terminating $_3F_2(2)$ with applications}, Bull. Korean Math. Soc. \textbf{49} (2012), no.~3, 621\nobreakdash--633. DOI: \url{https://doi.org/10.4134/BKMS.2012.49.3.621}.

\bibitem{15}
Y. S. Kim, M. A. Rakha, and A. K. Rathie, \textit{Generalization of Kummer's second theorem with applications}, translated from Zh. Vychisl. Mat. Mat. Fiz. \textbf{50} (2010), no.~3, 407\nobreakdash--422; Comptut. Math. Math. Phys. \textbf{50} (2010), no.~3, 387\nobreakdash--402. DOI: \url{https://doi.org/10.1134/S0965542510030024}.

\bibitem{16}
Y. S. Kim and A. K. Rathie, \textit{Some results for terminating ${\,}_2F_1(2)$ series}, J. Inequ. Appl. \textbf{2013}, 2013:365, 12~pp. DOI: \url{https://doi.org/10.1186/1029-242X-2013-365}.

\bibitem{17}
Y. S. Kim, A. K. Rathie, and R. B. Paris, \textit{Evaluations of some terminating hypergeometric ${\,}_2F_1(2)$ series with applications}, Turkish J. Math. \textbf{42} (2018), no.~5, 2563\nobreakdash--2575. DOI: \url{https://doi.org/10.3906/mat-1804-67}.

\bibitem{18}
T. Koshy, \emph{Catalan Numbers with Applications}, Oxford University Press, Oxford, 2009.

\bibitem{19}
D. Lim, \textit{A note on a generalization of Riordan's combinatorial identity via a hypergeometric series approach}, Notes Number Theory Discrete Math. \textbf{29} (2023), no.~3, 421\nobreakdash--425. DOI: \url{https://doi.org/10.7546/nntdm.2023.29.3.421-425}.

\bibitem{21}
H. Prodinger, \textit{Knuth's old sum: A survey}, Bull. Eur. Assoc. Theor. Comput. Sci. \textbf{52} (1992), 232\nobreakdash--245.

\bibitem{20}
H. Prodinger, \textit{Some information about the binomial transform}, Fibonacci Quart. \textbf{32} (1994), no.~5, 412\nobreakdash--415.

\bibitem{22}
A. P. Prudnikov, Yu. A. Brychkov, and O. I. Marichev, \emph{Integrals and Series}, Vol.~3. \emph{More Special Functions}. Translated from the Russian by G. G. Gould. Gordon and Breach Science Publishers, New York, 1990.

\bibitem{Gauss-2Formula.tex}
F. Qi, \textit{Combinatorial identities derived from explicit formulas of Gauss hypergeometric functions}, arXiv:2607.10643. DOI: \url{https://doi.org/10.48550/arXiv.2607.10643}.

\bibitem{Catalan-Int-Surv.tex}
F. Qi and B.-N. Guo, \textit{Integral representations of the Catalan numbers and their applications}, Mathematics \textbf{5} (2017), no.~3, Art.~40, 31~pp. DOI: \url{https://doi.org/10.3390/math5030040}.

\bibitem{DA19034-cas-sc.tex}
F. Qi, C.-Y. He, and D. Lim, \textit{Explicit formulas of two Gauss hypergeometric functions and several combinatorial identities}, Discrete Appl. Math. \textbf{393} (2026), 215\nobreakdash--229. DOI: \url{https://doi.org/10.1016/j.dam.2026.06.023}.

\bibitem{ScienceAsia-2022-0169.tex}
F. Qi and D. Lim, \textit{Integral representations and properties of several finite sums containing central binomial coefficients}, ScienceAsia \textbf{49} (2023), no.~2, 205\nobreakdash--211. DOI: \url{https://doi.org/10.2306/scienceasia1513-1874.2022.137}.

\bibitem{arcsin-power-wei.tex}
F. Qi, D.-W. Niu, and D. Lim, \textit{Some combinatorial identities containing central binomial coefficients or Catalan numbers}, Appl. Math. Sci. Eng. \textbf{31} (2023), no.~1, Paper No.~2204233, 12~pp. DOI: \url{https://doi.org/10.1080/27690911.2023.2204233}.

\bibitem{23}
E. D. Rainville, \textit{Special Functions}, Reprint of 1960 first edition. Chelsea Publishing Co., Bronx, NY, 1971.

\bibitem{24}
A. K. Rathie, I. Kim, and R. B. Paris, \textit{A note on a generalization of two well known combinatorial identities via a hypergeometric series approach}, Integers \textbf{22} (2022), Paper No.~A28, 6 pp.

\bibitem{27}
A. K. Rathie and D. Lim, \textit{A note on generalization of combinatorial identities due to Gould and Touchard}, Axioms \textbf{12} (2023), no.~3, Paper No.~268, 4~pp. DOI: \url{https://doi.org/10.3390/axioms12030268}.

\bibitem{25}
A. K. Rathie and V. Nagar, \textit{On Kummer's second theorem involving product of generalized hypergeometric series}, Matematiche (Catania) \textbf{50} (1995), no.~1, 35\nobreakdash--38.

\bibitem{26}
A. K. Rathie and T. K. Pog\'any, \textit{New summation formula for ${\,}_3F_2\bigl(\frac12\bigr)$ and a Kummer-type II transformation of ${\,}_2F_2(x)$}, Math. Commun. \textbf{13} (2008), no.~1, 63\nobreakdash--66.

\bibitem{29}
J. Riordan, \textit{A note on Catalan parentheses}, Amer. Math. Monthly \textbf{80} (1973), 904\nobreakdash--906. DOI: \url{https://doi.org/10.2307/2319398}.

\bibitem{28}
J. Riordan, \emph{Combinatorial Identities}, Reprint of the 1968 original, Robert E. Krieger Publishing Co., Huntington, N.Y., 1979.

\bibitem{31}
L. W. Shapiro, \textit{A short proof of an identity of Touchard's concerning Catalan numbers}, J. Combinatorial Theory Ser. A \textbf{20} (1976), no.~3, 375\nobreakdash--376. DOI: \url{https://doi.org/10.1016/0097-3165(76)90034-0}.

\bibitem{30}
L. W. Shapiro, \textit{Catalan numbers and total information numbers}, Proceedings of the Sixth Southeastern Conference on Combinatorics, Graph Theory and Computing (Florida Atlantic Univ., Boca Raton, Fla., 1975), pp.~531\nobreakdash--539, Congress. Numer., No. XIV, Utilitas Math., Winnipeg, MB, 1975.

\bibitem{Temme-96-book}
N. M. Temme, \emph{Special Functions: An Introduction to Classical Functions of Mathematical Physics}, A Wiley-Interscience Publication, John Wiley \& Sons, Inc., New York, 1996. DOI: \url{https://doi.org/10.1002/9781118032572}.

\bibitem{33}
J. Touchard, \textit{Sur certaines \'equations fontionnelles}, Proc. Int. Math. Congress, Toronto, 1924. \textbf{1} (1928), 465\nobreakdash--472.

\end{thebibliography}
\end{document}